\documentclass[psamsfont,a4paper,12pt]{amsart}

\usepackage[english]{babel}   %loads english
\numberwithin{equation}{section}

\usepackage{setspace}     %controls line-spacing
\usepackage{enumerate} 	%allows lists
\usepackage{afterpage} 	%keeps tables and figures where they belong
\usepackage[usenames,dvipsnames]{color}
\usepackage{graphicx}	%allows inclusion of pictures

\usepackage{amsmath} 	%just apparently the most important thing ever
\usepackage{amssymb}	%symbols in math
\usepackage{mathrsfs} 	%really fancy script font
\usepackage{amsfonts} 	%standard fonts
\usepackage{latexsym} 	%more symbols
\usepackage[latin1]{inputenc} %sets up proper special characters

\usepackage{amsthm} 	%sets up theorem styling
\usepackage{stackrel}       %allows stacking in mathmode  
\usepackage{amscd} 	%commutative diagrams                                 

\usepackage{tabularx}	%automatically widening tables
\usepackage{longtable}	%for tables longer than a page

\usepackage{verbatim} %comment environment?
\usepackage{blindtext} %something about dummy text

\usepackage{amssymb,latexsym,amsmath}    % Standard packages
\usepackage{pstricks,multido,pst-plot}		 % Packages for drawing figures
\usepackage{multicol,fancyheadings}			   % Misc.
\usepackage{float}                       % Used to define Algorithm enviro
\usepackage[vcentermath,enableskew]{youngtab}

\usepackage{subfig}

\allowdisplaybreaks

\def\a{\alpha}

\newcommand{\C}{\mathbb{C}}

\newcommand{\Z}{\mathbb{Z}}

\newtheorem{theorem}{Theorem}[section]  %%with section numbering
\newtheorem{proposition}[theorem]{Proposition}

\newtheorem{corollary}[theorem]{Corollary}

\newtheorem{Thm}{Theorem}[]		%%with numbering

\newtheorem*{thm}{Theorem}	%%without numbering

\theoremstyle{definition}

\newtheorem{example}{Example}[section]
\newtheorem*{exs}{Examples}

\theoremstyle{remark}
 
\newtheorem*{rmk}{Remark}

\newtheorem{ind}[]{{\rm\it Indice}}

\usepackage{multicol}

\usepackage{tikz}
\usetikzlibrary{arrows}

\usepackage{paralist}
\usepackage{cite}

\title{CM Evaluations of the Goswami-Sun series}

\author[Dawsey]{Madeline Locus Dawsey*}
\address{Department of Mathematics and Computer Science,
Emory University, Atlanta, GA 30322}
\email{madeline.locus@emory.edu}

\author[Ono]{Ken Ono}
\address{Department of Mathematics and Computer Science, Emory University, Atlanta, GA 30322}
\email{ono@mathcs.emory.edu}

\begin{document}

\thanks{*This author was previously known as Madeline Locus.}

\begin{abstract}
In recent work, Sun constructed two $q$-series, and he showed that their limits as $q\rightarrow1$ give new derivations of the Riemann-zeta values $\zeta(2)=\pi^2/6$ and $\zeta(4)=\pi^4/90$.  Goswami extended these series to an infinite family of $q$-series, which he analogously used to obtain new derivations of the evaluations of $\zeta(2k)\in\mathbb{Q}\cdot\pi^{2k}$ for every positive integer $k$.  Since it is well known that $\Gamma\left(\frac{1}{2}\right)=\sqrt{\pi}$, it is natural to seek further specializations of these series which involve special values of the $\Gamma$-function.  Thanks to the theory of complex multiplication, we show that the values of these series at all CM points $\tau$, where $q:=e^{2\pi i\tau}$, are algebraic multiples of specific ratios of $\Gamma$-values.  In particular, classical formulas of Ramanujan allow us to explicitly evaluate these series as algebraic multiples of powers of  $\Gamma\left(\frac{1}{4}\right)^4/\pi^3$ when $q=e^{-\pi}$, $e^{-2\pi}$.
\end{abstract}

\maketitle

\section{Introduction and Statement of Results}\label{1}

Recently, Sun \cite{S} obtained two $q$-series identities which allowed him to prove that
\begin{equation}\label{sun1}
\lim_{\substack{q\rightarrow1\\|q|<1}}(1-q)^2\sum_{n=0}^\infty\frac{q^n(1+q^{2n+1})}{(1-q^{2n+1})^2}=\frac{3}{2}\zeta(2)=\frac{\pi^2}{4}
\end{equation}
and
\begin{equation}\label{sun2}
\lim_{\substack{q\rightarrow1\\|q|<1}}(1-q)^4\sum_{n=0}^\infty\frac{q^{2n}(1+4q^{2n+1}+q^{4n+2})}{(1-q^{2n+1})^4}=\frac{45}{8}\zeta(4)=\frac{\pi^4}{16}.
\end{equation}
Sun's formulas lead to the natural question: Are these $q$-series a glimpse of an infinite family that offers new derivations for the evaluations of $\zeta(2k)$ for all positive integers $k$?  Goswami elegantly answered this problem in \cite{G}; he defined a natural family of identities whose limits as $q\rightarrow1$ with $|q|<1$ give Euler's formula for the Riemann-zeta values at all even integers.

These results have been described as $q$-analogues of Euler's identities for $\zeta(2k)$. Here we offer further support of this view. Namely, to be a strong $q$-analogue, one hopes for further specializations of $q$ which are expressions in related special functions.  We address this question by observing that $\zeta(2k)\in\mathbb{Q}\cdot\pi^{2k}=\mathbb{Q}\cdot\Gamma\left(\frac{1}{2}\right)^{4k}$, and we ask if Goswami's series have evaluations involving algebraic multiples of naturally corresponding $\Gamma$-values.  We show that this is indeed the case, thanks to the theory of complex multiplication and modular forms.

In order to state our results, we first recall the $q$-series that Goswami assembled which extended Sun's original identities into an infinite family.  Throughout, $k$ is a positive integer.  If we denote the Stirling numbers of the second kind by $\left\{\begin{smallmatrix}n \\ k\end{smallmatrix}\right\}$, then we define $a_k(m)$ and $b_k(\ell)$ by
\begin{eqnarray*}
a_k(m)&:=&\sum_{j=0}^{2k-1}j!(-1)^j\left\{\begin{smallmatrix}2k-1 \\ j\end{smallmatrix}\right\}\left(\begin{smallmatrix}j \\ m\end{smallmatrix}\right),\\
\hspace{.5cm}b_k(\ell)&:=&\sum_{m=0}^{2k-1}(-1)^ma_k(m)\left(\begin{smallmatrix}2k-m-1 \\ \ell\end{smallmatrix}\right)\in\mathbb{Z}.
\end{eqnarray*}
Using these quantities, we define the degree $2k-2$ polynomial
\begin{equation*}
P_{2k-2}^e(z):=\sum_{\ell=1}^{2k-1}(-1)^\ell b_k(\ell)z^{\ell-1},
\end{equation*}
and the degree $4k-2$ polynomial
\begin{equation*}
P_{4k-2}^o(z):=(1+z)^{2k}P_{2k-2}^e(z)-2^{2k-1}zP_{2k-2}^e\left(z^2\right).
\end{equation*}
For notational convenience, we define Goswami's $q$-series as follows.
\begin{equation}
\label{G}\mathcal{G}_{2k}(q):=\left\{\begin{array}{ll}\sum\limits_{n=0}^\infty\dfrac{q^{2n+1}P_{4k-2}^o\left(q^{2n+1}\right)}{\left(1-q^{4n+2}\right)^{2k}},&\mbox{if }k\text{ is odd.}\vspace{.5cm}\\
2^{2k-1}\sum\limits_{n=0}^\infty\dfrac{q^{4n+2}P_{2k-2}^e\left(q^{4n+2}\right)}{\left(1-q^{4n+2}\right)^{2k}},&\mbox{if }k\text{ is even,}\end{array}\right.
\end{equation}
\begin{rmk}
When $k=1$ and $k=2$, these are essentially Sun's $q$-series.  A critical feature of the results obtained here is that the $\mathcal{G}_{2k}(q)$ are holomorphic modular forms on $\Gamma_0(4)$ of integer weight $2k$.
\end{rmk}
As usual, we let $\overline{\mathbb{Q}}$ denote the algebraic closure of the field of rational numbers.  Suppose that $D<0$ is the fundamental discriminant of the imaginary quadratic field $\mathbb{Q}(\sqrt{D})$.  Let $h(D)$ denote the class number of $\mathbb{Q}(\sqrt{D})$, and define $h'(D):=1/3$ (resp. $1/2$) when $D=-3$ (resp. $-4$), and $h'(D):=h(D)$ when $D<-4$.  We then let
\begin{equation}\label{omega}
\omega_D:=\frac{1}{\sqrt{\pi}}\left(\prod_{j=1}^{|D|-1}\Gamma\left(\frac{j}{|D|}\right)^{\chi_D(j)}\right)^{\frac{1}{2h'(D)}},
\end{equation}
where $\chi_D(\bullet):=\left(\frac{D}{\bullet}\right)$.  In terms of this notation, we obtain the following theorem.

\begin{theorem}\label{main}
If $D<0$ is a fundamental discriminant and $\tau\in\mathbb{H}\cap\mathbb{Q}(\sqrt{D})$, then
\begin{equation*}
\mathcal{G}_{2k}\left(e^{2\pi i\tau}\right)\in\overline{\mathbb{Q}}\cdot\omega_D^{2k}.
\end{equation*}
\end{theorem}

Thanks to classical formulas of Ramanujan \cite{lost}, it is simple to explicitly evaluate $\mathcal{G}_{2k}\left(e^{-\pi}\right)$ and $\mathcal{G}_{2k}\left(e^{-2\pi}\right)$.  To make this precise, we define the rational number\footnote{In \cite{G}, Goswami refers to $\mathcal{Z}(2k)$ as $d_k$.  We use $\mathcal{Z}(2k)$ to emphasize that these numbers are simple rational multiples of $\zeta(2k)/\pi^{2k}$.}
\begin{equation}\label{dk}
\mathcal{Z}(2k):=-\frac{(-16)^kB_{2k}\left(4^k-1\right)}{8k}=4^{k-1}\left(4^k-1\right)(2k)!\cdot\frac{\zeta(2k)}{\pi^{2k}},
\end{equation}
where $B_{2k}$ is the index $2k$ Bernoulli number.  Furthermore, we let $(a;q)_\infty:=\prod_{n\geq0}\left(1-aq^n\right)$ denote the usual infinite $q$-Pochhammer symbol.  If $k\geq2$, then define $\a_{2k}(1),\dots,\a_{2k}(k-1)$ to be the unique rational numbers satisfying
\begin{equation}\label{eta-quot1}
\sum_{j=1}^{k-1}\a_{2k}(j)\cdot\frac{q^j\left(q^4;q^4\right)_\infty^{16j}(q;q)_\infty^{8j}}{\left(q^2;q^2\right)_\infty^{24j}}=\left(\mathcal{G}_{2k}(q)-\mathcal{Z}(2k)\cdot\frac{q^k\left(q^4;q^4\right)_\infty^{8k}}{\left(q^2;q^2\right)_\infty^{4k}}\right)\cdot\frac{(q;q)_\infty^{8k}\left(q^4;q^4\right)_\infty^{8k}}{\left(q^2;q^2\right)_\infty^{20k}}.
\end{equation}
Since the $j$th summand on the left is of the form $\a_{2k}(j)q^j+O\left(q^{j+1}\right)$, the $\a_{2k}(j)$ are easily computed by diagonalization.  In the case where $k=1$, there simply are no $\a_{2k}(j)$ numbers.  In terms of this notation, we obtain the following corollary.

\begin{corollary}\label{cor}
If $k$ is a positive integer and $a:=\sqrt{2}-1$, then
\begin{eqnarray*}
\mathcal{G}_{2k}\left(e^{-\pi}\right)&=&\left(\frac{\mathcal{Z}(2k)}{2^{7k}}+\frac{1}{2^{2k}}\sum_{j=1}^{k-1}\frac{\a_{2k}(j)}{2^{5j}}\right)\cdot\left(\frac{\Gamma\left(\frac{1}{4}\right)^4}{\pi^3}\right)^k,\\
\mathcal{G}_{2k}\left(e^{-2\pi}\right)&=&\left(\frac{\mathcal{Z}(2k)a^{2k}}{2^{9k}}+\frac{1}{2^{5k}a^{2k}}\sum_{j=1}^{k-1}\frac{\a_{2k}(j)a^{4j}}{2^{4j}}\right)\cdot\left(\frac{\Gamma\left(\frac{1}{4}\right)^4}{\pi^3}\right)^k.
\end{eqnarray*}
\end{corollary}

\begin{exs}
Here we illustrate Corollary \ref{cor} for $k=3$ and 4.  If $k=3$, then we have that
\begin{eqnarray*}
\mathcal{G}_6\left(e^{-\pi}\right)&=&\left(\frac{\mathcal{Z}(6)}{2^{21}}+\frac{1}{2^{12}}\right)\cdot\left(\frac{\Gamma\left(\frac{1}{4}\right)^4}{\pi^3}\right)^3,\\
\mathcal{G}_6\left(e^{-2\pi}\right)&=&\left(\frac{\mathcal{Z}(6)\left(\sqrt{2}-1\right)^6}{2^{27}}+\frac{1-\left(\sqrt{2}-1\right)^4}{2^{19}\left(\sqrt{2}-1\right)^2}\right)\cdot\left(\frac{\Gamma\left(\frac{1}{4}\right)^4}{\pi^3}\right)^3.
\end{eqnarray*}
If $k=4$, then we have that
\begin{eqnarray*}
\mathcal{G}_8\left(e^{-\pi}\right)&=&\left(\frac{\mathcal{Z}(8)}{2^{28}}+\frac{1}{2^{12}}\right)\cdot\left(\frac{\Gamma\left(\frac{1}{4}\right)^4}{\pi^3}\right)^4,\\
\mathcal{G}_8\left(e^{-2\pi}\right)&=&\left(\frac{\mathcal{Z}(8)\left(\sqrt{2}-1\right)^8}{2^{36}}+\frac{1-\left(\sqrt{2}-1\right)^4}{2^{21}}\right)\cdot\left(\frac{\Gamma\left(\frac{1}{4}\right)^4}{\pi^3}\right)^4.
\end{eqnarray*}
These examples will be explained further in Section \ref{exs}.
\end{exs}

This paper is organized as follows.  In Section \ref{G-S}, we recall the Goswami-Sun identities and the relation between $\mathcal{G}_{2k}(q)$ and modular forms.  In Section \ref{mf}, we recall essential facts about modular forms, and in Section \ref{proof}, we use these results to prove Theorem \ref{main} and Corollary \ref{cor}.  In Section \ref{exs}, we conclude with a discussion of the examples given above.

\section*{Acknowledgements}
We thank Krishnaswami Alladi and Ankush Goswami for their beautiful ideas and contributions.  We also thank Zhi-Wei Sun for inspiring this work.

\section{The Goswami-Sun Identities}\label{G-S}
We now recall Goswami's work.  Let $T_n=n(n+1)/2$ denote the $n$th triangular number, and define the generating function of $T_n$ to be $$\psi(q):=\sum_{n\geq0}q^{T_n}.$$  Then Goswami \cite[Theorems~3.1~and~3.2]{G} proves the following theorem.
\begin{theorem}\label{Gos}
For any positive integer $k$, we have that
\begin{equation*}
T_{2k}(\tau):=\mathcal{G}_{2k}(q)-\mathcal{Z}(2k)\cdot q^k\psi\left(q^2\right)^{4k}
\end{equation*}
is the Fourier expansion of a weight $2k$ cusp form on $\Gamma_0(4)$, where $q:=e^{2\pi i\tau}$ and $\tau\in\mathbb{H}$.
\end{theorem}

In Section \ref{mf}, we apply the theory of modular forms and complex multiplication to study $T_{2k}(\tau)$ and $q^k\psi\left(q^2\right)^{4k}$ at all CM points.

\section{Some Facts about Modular Forms}\label{mf}

Here we recall some basic facts about modular forms.

\subsection{CM Values of Modular Forms}

In Goswami's work \cite{G}, the limit of the $q$-series identity in Theorem \ref{Gos} as $q\rightarrow1$ gives the constant term of a weight $2k$ Eisenstein series, which is described in terms of $\zeta$-values.  Our work depends on the values of modular forms at CM points.

Classically, the Chowla-Selberg formula \cite{CS} was developed in order to evaluate the Dedekind eta-function $$\eta(\tau):=q^{1/24}\prod_{n=1}^\infty\left(1-q^n\right)$$ at CM points whose discriminants are fundamental discriminants.  This was refined by van der Poorten and Williams \cite[Theorem 9.3]{PW}, who gave a closed formula for values of $\eta(\tau)$ in which $\tau$ is still required to be a CM point whose discriminant is fundamental.  More generally, we have the following theorem (for example, see p. 84 of \cite{123}) regarding evaluations of all modular forms at all CM points.

\begin{theorem}\label{Omega}
Suppose that $D<0$ is the fundamental discriminant of the imaginary quadratic field $\mathbb{Q}(\sqrt{D})$.  Then the number $\Omega_D\in\mathbb{C}^*$ defined by
\begin{equation*}
\Omega_D:=\frac{1}{\sqrt{2\pi|D|}}\left(\prod_{j=1}^{|D|-1}\Gamma\left(\frac{j}{|D|}\right)^{\chi_D(j)}\right)^{\frac{1}{2h'(D)}}
\end{equation*}
has the property that $f(\tau)\in\overline{\mathbb{Q}}\cdot\Omega_D^k$ for all $\tau\in\mathbb{H}\cap\mathbb{Q}(\sqrt{D}),$ all $k\in\mathbb{Z}$, and all modular forms $f$ of weight $k$ with algebraic Fourier coefficients.
\end{theorem}

In the special cases of the CM points $\tau\in\{i/2,i,2i,4i\}$, Theorem \ref{Omega} can be made explicit using the following formulas of Ramanujan (see p. 326 of \cite{lost}),
\begin{align*}
f\left(-e^{-\pi}\right)&=\frac{\pi^{\frac{1}{4}}e^{\frac{\pi}{24}}}{2^{\frac{3}{8}}\Gamma\left(\frac{3}{4}\right)},\\
f\left(-e^{-2\pi}\right)&=\frac{\pi^{\frac{1}{4}}e^{\frac{\pi}{12}}}{2^{\frac{1}{2}}\Gamma\left(\frac{3}{4}\right)},\\
f\left(-e^{-4\pi}\right)&=\frac{\pi^{\frac{1}{4}}e^{\frac{\pi}{6}}}{2^{\frac{7}{8}}\Gamma\left(\frac{3}{4}\right)},\\
f\left(-e^{-8\pi}\right)&=\frac{\pi^{\frac{1}{4}}\left(\sqrt{2}-1\right)^{\frac{1}{4}}e^{\frac{\pi}{3}}}{2^{\frac{21}{16}}\Gamma\left(\frac{3}{4}\right)},
\end{align*}
where $f(-q):=\prod_{n\geq1}\left(1-q^n\right)$.  The above formulas can be rewritten in terms of the Dedekind eta-function by noticing that $\eta(\tau)=q^{1/24}f(-q)$.  By applying the functional equation of the $\Gamma$-function, namely $\Gamma(1-z)\Gamma(z)=\pi/\sin(\pi z)$ for $z\not\in\mathbb{Z}$, in terms of $$\Omega_{-4}=\frac{1}{2\sqrt{2\pi}}\cdot\frac{\Gamma\left(\frac{1}{4}\right)}{\Gamma\left(\frac{3}{4}\right)},$$ we obtain
\begin{equation}\label{etas}
\eta(i/2)=2^{\frac{1}{8}}\cdot\Omega_{-4}^{\frac{1}{2}},\hspace{.5cm}\eta(i)=\Omega_{-4}^{\frac{1}{2}},\hspace{.5cm}\eta(2i)=\frac{1}{2^{\frac{3}{8}}}\cdot\Omega_{-4}^{\frac{1}{2}},\hspace{.5cm}\eta(4i)=\frac{\left(\sqrt{2}-1\right)^{\frac{1}{4}}}{2^{\frac{13}{16}}}\cdot\Omega_{-4}^{\frac{1}{2}}.
\end{equation}
We shall make use of these formulas to prove Corollary \ref{cor}.

\subsection{Modular Forms on $\Gamma_0(4)$}\label{mfs}

Here we recall standard facts about modular forms on $\Gamma_0(4)$.  Recall that the theta function given by
\begin{equation*}
\theta(\tau):=\sum_{n=-\infty}^\infty q^{n^2}
\end{equation*}
is a weight $\frac{1}{2}$ modular form on $\Gamma_0(4)$, and that the  weight 2 Eisenstein series
\begin{equation*}
F(\tau):=\sum_{n=0}^\infty\sigma_1(2n+1)q^{2n+1}
\end{equation*}
is a modular form on $\Gamma_0(4)$ as well, where $\sigma_1(n)$ denotes the sum of the positive divisors of $n$.  It is known that every modular form on $\mathrm{SL}_2(\mathbb{Z})$ and $\Gamma_0(4)$ can be expressed as a rational function in $\eta(\tau)$, $\eta(2\tau)$, and $\eta(4\tau)$ (see \cite[Theorem~1.67]{Web} for $\mathrm{SL}_2(\mathbb{Z})$ and \cite{O} for $\Gamma_0(4)$).  In the case of $\Gamma_0(4)$, this fact relies on the observation that $F(\tau)$ and $\theta(\tau)$ are given in terms of Dedekind eta-quotients in the following way:
\begin{equation}\label{etaquotientforms}
\theta(\tau)=\frac{\eta(2\tau)^5}{\eta(\tau)^2\eta(4\tau)^2},\hspace{1cm}F(\tau)=\frac{\eta(4\tau)^8}{\eta(2\tau)^4}.
\end{equation}

It is also very well known that the two Eisenstein series $E_4(\tau)$ and $E_6(\tau)$ generate the algebra of all modular forms on $\mathrm{SL}_2(\mathbb{Z})$ (for example, see \cite[Theorem~1.23]{Web}).  The analogous statement for modular forms on $\Gamma_0(4)$ involves the forms $F(\tau)$ and $\theta(\tau)$.  Namely, the following complete description of the spaces $M_k\left(\Gamma_0(4),\psi_k\right)$ for $k\in\frac{1}{2}\mathbb{N}$ and $$\psi_k:=\begin{cases}\chi_0,&\mbox{if }k\in2\mathbb{Z}\text{ or }k\in\frac{1}{2}+\mathbb{Z}, \\ \left(\frac{-4}{\bullet}\right),&\mbox{if }k\in1+2\mathbb{Z},\end{cases}$$ where $\chi_0$ is the trivial character, is proved in \cite{C,K}.  As a graded algebra, we have that
\begin{equation*}
\bigoplus_{k\in\frac{1}{2}\mathbb{Z}}M_k\left(\Gamma_0(4),\psi_k\right)\cong\mathbb{C}[F,\theta].
\end{equation*}
Moreover, we have the following proposition (see \cite[Corollary 3.3]{O}) describing canonical representations of modular forms on $\Gamma_0(4)$ in terms of $F(\tau)$ and $\theta(\tau)$.

\begin{proposition}\label{alphas}
If $k\in\frac{1}{2}\mathbb{N}$, then each $f(\tau)\in M_k\left(\Gamma_0(4),\psi_k\right)$ has a unique expansion in terms of $F(\tau)$ and $\theta(\tau)$ of the form
\begin{equation}\label{basis}
f(\tau)=\sum_{j=0}^{[k/2]}\a_k(j)F(\tau)^j\theta(\tau)^{2k-4j}.
\end{equation}
Moreover, $f(\tau)$ is a cusp form if and only if the coefficients $\a_k(j)$ satisfy:
\begin{enumerate}[\hspace{.5cm}(i)]
\item $\a_k(0)=0$,\\
\item $\a_{k}(k/2)=0$ when $k\in2\mathbb{Z}$, and\\
\item $\sum\limits_{j=0}^{[k/2]}\a_k(j)\left(\frac{1}{16}\right)^j=0$.
\end{enumerate}
\end{proposition}
Combining the above decomposition in terms of $F(\tau)$ and $\theta(\tau)$ for the cusp form $T_{2k}(\tau)$ with the eta-quotients in (\ref{etaquotientforms}), we obtain new expressions for the series $\mathcal{G}_{2k}(q)$ in Section \ref{proof}.

\section{Proof of Theorem \ref{main} and Corollary \ref{cor}}\label{proof}

We may write the product on the right hand side of Theorem \ref{Gos} in terms of eta-quotients as follows:
\begin{equation}
\label{eta}\mathcal{Z}(2k)\cdot q^k\psi\left(q^2\right)^{4k}=\mathcal{Z}(2k)\cdot q^k\prod_{n=1}^\infty\frac{\left(1-q^{4n}\right)^{4k}}{\left(1-q^{4n-2}\right)^{4k}}=\mathcal{Z}(2k)\cdot\frac{\eta(4\tau)^{8k}}{\eta(2\tau)^{4k}}.
\end{equation}
Now, by (\ref{basis})\footnote{The weights in Theorem 
\ref{main} are $2k$ as opposed to $k$ in the section above.} we can express the cusp form $T_{2k}(\tau)$ as
\begin{equation}\label{Tinbasis}
T_{2k}(\tau)=\sum_{j=0}^k\a_{2k}(j)F(\tau)^j\theta(\tau)^{4k-4j},
\end{equation}
and by (\ref{etaquotientforms}) we can write (\ref{Tinbasis}) in terms of eta-quotients in the following way:
\begin{equation*}
T_{2k}(\tau)=\frac{\eta(2\tau)^{20k}}{\eta(\tau)^{8k}\eta(4\tau)^{8k}}\sum_{j=0}^k\a_{2k}(j)\cdot\frac{\eta(4\tau)^{16j}\eta(\tau)^{8j}}{\eta(2\tau)^{24j}}.
\end{equation*}
By Proposition \ref{alphas} \textit{(i)} and \textit{(ii)}, we may simplify this expression to
\begin{equation}\label{Ts}
T_{2k}(\tau)=\frac{\eta(2\tau)^{20k}}{\eta(\tau)^{8k}\eta(4\tau)^{8k}}\sum_{j=1}^{k-1}\a_{2k}(j)\cdot\frac{\eta(4\tau)^{16j}\eta(\tau)^{8j}}{\eta(2\tau)^{24j}}.
\end{equation}
Now, combining (\ref{eta}) with (\ref{Ts}), we see that the series $\mathcal{G}_{2k}(q)$ can be expressed as
\begin{equation}\label{eta-quot}
\mathcal{G}_{2k}(q)=\mathcal{Z}(2k)\cdot\frac{\eta(4\tau)^{8k}}{\eta(2\tau)^{4k}}+\frac{\eta(2\tau)^{20k}}{\eta(\tau)^{8k}\eta(4\tau)^{8k}}\sum_{j=1}^{k-1}\a_{2k}(j)\cdot\frac{\eta(4\tau)^{16j}\eta(\tau)^{8j}}{\eta(2\tau)^{24j}}.
\end{equation}
From the above expression for $\mathcal{G}_{2k}(q)$, it is clear that if $k\geq1$, then the $\a_{2k}(j)$ are the unique rational numbers such that
\begin{equation*}
\sum_{j=1}^{k-1}\a_{2k}(j)\cdot\frac{\eta(4\tau)^{16j}\eta(\tau)^{8j}}{\eta(2\tau)^{24j}}=\left(\mathcal{G}_{2k}(q)-\mathcal{Z}(2k)\cdot\frac{\eta(4\tau)^{8k}}{\eta(2\tau)^{4k}}\right)\cdot\frac{\eta(\tau)^{8k}\eta(4\tau)^{8k}}{\eta(2\tau)^{20k}}.
\end{equation*}
This implies the definition for the $\a_{2k}(j)$ in (\ref{eta-quot1}).

Theorem \ref{Omega} along with (\ref{eta-quot}) immediately imply that evaluations of the Goswami-Sun series at CM points $\tau\in\mathbb{H}\cap\mathbb{Q}(\sqrt{D})$ give values in $\overline{\mathbb{Q}}\cdot\Omega_D^{2k}$, which proves Theorem \ref{main} because $$\omega_D^{2k}=2^k|D|^k\cdot\Omega_D^{2k}.$$\qed

\begin{proof}[Proof of Corollary \ref{cor}]
If $D=-4$, then $\mathbb{Q}(\sqrt{D})=\mathbb{Q}(i)$ and by (\ref{omega}) we have $$\omega_{-4}=\frac{1}{\sqrt{\pi}}\cdot\frac{\Gamma\left(\frac{1}{4}\right)}{\Gamma\left(\frac{3}{4}\right)}.$$  We apply the functional equation of the $\Gamma$-function to rewrite $\omega_{-4}$ as $$\omega_{-4}=\frac{\Gamma\left(\frac{1}{4}\right)^2}{\sqrt{2}\pi^{3/2}}.$$  In particular, applying the values of $\eta(\tau)$ in (\ref{etas}) to all of the eta-functions in (\ref{eta-quot}), we evaluate $\mathcal{G}_{2k}\left(e^{2\pi i\tau}\right)$ at $\tau=\frac{i}{2}$ and $\tau=i$ to obtain the values in Corollary \ref{cor}.
\end{proof}

\section{Examples}\label{exs}

\begin{example}\label{3}
If $k=3$, then Theorem \ref{main} becomes
\begin{equation*}
\mathcal{G}_6(q)=\mathcal{Z}(6)\cdot\frac{\eta(4\tau)^{24}}{\eta(2\tau)^{12}}+\frac{\eta(2\tau)^{60}}{\eta(\tau)^{24}\eta(4\tau)^{24}}\left(\a_6(1)\cdot\frac{\eta(4\tau)^{16}\eta(\tau)^8}{\eta(2\tau)^{24}}+\a_6(2)\cdot\frac{\eta(4\tau)^{32}\eta(\tau)^{16}}{\eta(2\tau)^{48}}\right),
\end{equation*}
and we calculate that $\a_6(1)=1$ and $\a_6(2)=-16$.  Then we have that
\begin{equation*}
\mathcal{G}_6(q)=\mathcal{Z}(6)\cdot\frac{\eta(4\tau)^{24}}{\eta(2\tau)^{12}}+\frac{\eta(2\tau)^{60}}{\eta(\tau)^{24}\eta(4\tau)^{24}}\left(\frac{\eta(4\tau)^{16}\eta(\tau)^8}{\eta(2\tau)^{24}}-16\cdot\frac{\eta(4\tau)^{32}\eta(\tau)^{16}}{\eta(2\tau)^{48}}\right).
\end{equation*}
Corollary \ref{cor} in this case gives
\begin{equation*}
\mathcal{G}_6\left(e^{-\pi}\right)=\left(\frac{\mathcal{Z}(6)}{2^{21}}+\frac{1}{2^{12}}\right)\cdot\left(\frac{\Gamma\left(\frac{1}{4}\right)^4}{\pi^3}\right)^3=0.0633804556\ldots
\end{equation*}
and
\begin{equation*}
\mathcal{G}_6\left(e^{-2\pi}\right)=\left(\frac{\mathcal{Z}(6)\left(\sqrt{2}-1\right)^6}{2^{27}}+\frac{1-\left(\sqrt{2}-1\right)^4}{2^{19}\left(\sqrt{2}-1\right)^2}\right)\cdot\left(\frac{\Gamma\left(\frac{1}{4}\right)^4}{\pi^3}\right)^3=0.0018690318\ldots.
\end{equation*}
\end{example}

\begin{example}\label{4}
If $k=4$, then Theorem \ref{main} becomes
\begin{align*}
\mathcal{G}_8(&q)=\mathcal{Z}(8)\cdot\frac{\eta(4\tau)^{32}}{\eta(2\tau)^{16}}\\
&+\frac{\eta(2\tau)^{80}}{\eta(\tau)^{32}\eta(4\tau)^{32}}\left(\a_8(1)\cdot\frac{\eta(4\tau)^{16}\eta(\tau)^8}{\eta(2\tau)^{24}}+\a_8(2)\cdot\frac{\eta(4\tau)^{32}\eta(\tau)^{16}}{\eta(2\tau)^{48}}+\a_8(3)\cdot\frac{\eta(4\tau)^{48}\eta(\tau)^{24}}{\eta(2\tau)^{72}}\right),
\end{align*}
and we calculate that $\a_8(1)=0$, $\a_8(2)=128,$ and $\a_8(3)=-2048$.  Then we have that
\begin{equation*}
\mathcal{G}_8(q)=\mathcal{Z}(8)\cdot\frac{\eta(4\tau)^{32}}{\eta(2\tau)^{16}}+\frac{\eta(2\tau)^{80}}{\eta(\tau)^{32}\eta(4\tau)^{32}}\left(128\cdot\frac{\eta(4\tau)^{32}\eta(\tau)^{16}}{\eta(2\tau)^{48}}-2048\cdot\frac{\eta(4\tau)^{48}\eta(\tau)^{24}}{\eta(2\tau)^{72}}\right).
\end{equation*}
Corollary \ref{cor} in this case gives
\begin{equation*}
\mathcal{G}_8\left(e^{-\pi}\right)=\left(\frac{\mathcal{Z}(8)}{2^{28}}+\frac{1}{2^{12}}\right)\cdot\left(\frac{\Gamma\left(\frac{1}{4}\right)^4}{\pi^3}\right)^4=0.2980189122\ldots
\end{equation*}
and
\begin{equation*}
\mathcal{G}_8\left(e^{-2\pi}\right)=\left(\frac{\mathcal{Z}(8)\left(\sqrt{2}-1\right)^8}{2^{36}}+\frac{1-\left(\sqrt{2}-1\right)^4}{2^{21}}\right)\cdot\left(\frac{\Gamma\left(\frac{1}{4}\right)^4}{\pi^3}\right)^4=0.0004465790\ldots.
\end{equation*}
\end{example}

\end{document}